\documentclass{amsart}
\usepackage{amsmath,amsthm,amscd,amssymb, color}
\usepackage{latexsym}
\usepackage[colorlinks, citecolor = blue]{hyperref}
\usepackage[alphabetic]{amsrefs}
\usepackage[capitalize]{cleveref}

\newcommand{\Q}{\mathbb Q}

\newcommand{\Z}{\mathbb Z}
\renewcommand{\phi}{\varphi}
\newcommand{\eps}{\varepsilon}
\newcommand{\calO}{\mathcal O}

\newcommand{\bmx}{\left( \begin{matrix}}
\newcommand{\emx}{\end{matrix} \right)}

\newcommand{\new}{\mathrm{new}}
\newcommand{\odd}{\mathrm{odd}}
\renewcommand{\mod}{\, \, \mathrm{mod} \, \,}

\newcommand{\leg}{\overwithdelims ()}

\DeclareMathOperator{\GL}{GL}
\DeclareMathOperator{\tr}{tr} 

\newtheorem{lem}{Lemma}
\numberwithin{lem}{section}
\newtheorem{prop}[lem]{Proposition}
\newtheorem{thm}[lem]{Theorem}

\newtheorem{cor}[lem]{Corollary}

\numberwithin{rem}{section}
\theoremstyle{definition}

\numberwithin{equation}{section}

\begin{document}

\title{Refined dimensions of cusp forms, and equidistribution and bias of signs}

\author{Kimball Martin}
\address{Department of Mathematics, University of
Oklahoma, Norman, OK 73019}

\date{\today}

\maketitle

\begin{abstract}
We refine known dimension formulas for spaces of cusp forms
of squarefree level, determining the dimension of subspaces generated
by newforms both with prescribed global root numbers
and with prescribed local signs of Atkin--Lehner operators.  
This yields precise
results on the distribution of signs of global functional equations
and sign patterns of Atkin--Lehner eigenvalues, refining and generalizing
earlier results of Iwaniec, Luo and Sarnak.  In particular, we exhibit a
strict bias towards root number +1 and a phenomenon that sign patterns
are biased in the weight but perfectly equidistributed in the level.
Another consequence is lower bounds on the number of Galois orbits.
\end{abstract}

%
\section*{Introduction}
%

Let $S^\new_k(N)$ denote the new subspace of weight $k$
elliptic cusp forms on $\Gamma_0(N)$.
Dimensions of such spaces are well known (cf.\ \cite{martin}).  
If $N > 1$, one can decompose, in various ways, $S^\new_k(N)$ into certain natural
subspaces.  A crude decomposition is into the plus and minus spaces
$S^{\new, \pm}_k(N)$, which are the subspaces of $S^\new_k(N)$ generated by
newforms with global root number $\pm 1$, i.e., sign $\pm$ in the functional equation 
of their $L$-functions.  A more refined decomposition is to consider subspaces generated
by newforms with fixed local components $\pi_p$ at primes $p | N$, where each 
$\pi_p$ is a representation of $\GL_2(\Q_p)$ of conductor $p^{v_p(N)}$.

In this paper, we obtain explicit dimension formulas for both of these types of subspaces
in the case $N > 1$ is squarefree.  This case is simple because there are
only two possibilities for the local components $\pi_p$, the Steinberg representation and
its unramified quadratic twist, and $\pi_p$ is determined by the Atkin--Lehner eigenvalue.   
The proof relies on a trace formula for products of Atkin--Lehner operators on
$S_k(N)$ due to Yamauchi \cite{yamauchi}, which we translate to $S_k^\new(N)$ in \cref{sec:traces}.  It was already known that such dimensions can be computed in 
principle in this way, and some cases have been done before: the prime level case is in
\cite{wakatsuki}, asymptotics for dimensions of plus and minus spaces were given in \cite{ILS},
and \cite{hasegawa-hashimoto} gave a formula for the full cusp space in level 2 with prescribed
Atkin--Lehner eigenvalues.   So while the derivation is not especially novel,
 we hope the explicit formulas and their consequences (particularly the biases discussed below) 
 may be of interest.  In fact, our motiviation was different from \cite{hasegawa-hashimoto},
 \cite{ILS} and \cite{wakatsuki}, which all had mutually distinct motivations.

We emphasize that we are able to obtain quite simple formulas thanks to the
squarefree assumption.  The trace formula in \cite{yamauchi} is valid for arbitrary level, but becomes considerably more complicated.  In principle, our approach gives dimensions
of spaces with prescribed Atkin--Lehner eigenvalues in non-squarefree level also, but
the resulting formulas may be messy.  In any case, this would not give us dimensions for 
forms with specified local components $\pi_p$ for non-squarefree levels.

Our motivation in computing these dimensions comes from two sources.  First, this allows
us to get very precise results about the distribution of signs of global functional equations
and sign patterns for collections of Atkin--Lehner operators.  Various equidistribution results
are known about local components at unramified places, or equivalently, Hecke eigenvalues
at primes away from the level.  For instance, and perhaps most analogous,
distributions of signs of unramified Hecke eigenvalues are considered in \cite{KLSW}.
At ramified places, the most general results we know of for $\GL(2)$
are by Weinstein \cite{weinstein}, where he proves equidistribution of local inertia types for 
general level as the weight tends to infinity.  However, these local inertia types do not 
distinguish unramified quadratic twists, and thus give no information
in the case of squarefree level.  While the fact that equidistribution holds is not at
all surprising, the precise results we obtain about distribution and bias were perhaps
not expected.  Let us explain this in more detail.

\medskip
For the rest of the paper, assume $N > 1$ is squarefree and $k \ge 2$ is even.

In \cref{sec:pm} we obtain the dimension formulas for the plus and minus spaces.
This implies the root number is equidistributed
between $+1$ and $-1$ and the difference between the dimensions of the plus and 
minus spaces is essentially independent of $k$. It is also subpolynomial in $N$---precisely $O(2^{\omega(N)})$, where $\omega(N)$ is the number of prime divisors
of $N$.  
This is a considerable improvement upon
an earlier equidistribution result of Iwaniec--Luo--Sarnak \cite[(2.73)]{ILS} which
just implies the difference is $O((kN)^{5/6})$.
Moreover in any fixed space $S_k^\new(N)$, $+1$ always occurs at least as often as $-1$, i.e., there is a strict bias toward 
$+1$, and the size of this bias is on the order of the class number $h_{\Q(\sqrt{-N})}$.  
We initially found this bias surprising, but David Farmer explained to us how
such a bias (though perhaps not the size)
is actually predicted by the explicit formula.  We briefly explain this, and how
the arithmetic of quaternion algebras also suggest (in fact prove for $S_2(p)$) 
this bias.
Moreover, we determine for which fixed spaces $S_k^\new(N)$ there is perfect 
equidistribution (i.e., $+1$ occurs
exactly as often as $-1$): for $N=2, 3$ it merely depends on a congruence condition 
on $k$,
but for $N > 3$ it only happens twice, for $S_2^\new(37)$ and $S_2^\new(58)$.

Next, fix $M | N$ with $M > 1$.  In \cref{sec:sign-pat}, we 
look at distributions of sign patterns of Atkin--Lehner eigenvalues on $S_k^\new(N)$
for primes $p | M$.  Since the global root numbers are $(-1)^{k/2}$
times the product of the local signs, this is a refinement of looking
at distributions of root numbers.   We obtain simultaneous equidistribution of sign patterns in both the weight and level (fixing $M$ but varying $N$).
As with the case of root numbers, the error term in the asymptotic is constant when 
varying the weight and $O(2^{\omega(N)})$ when varying the level.  
We also find biases for certain sign patterns.  On a fixed
space $S_k^\new(N)$ there is a potential bias---which is toward or away from, depending on a 
parity condition---collections of signs being $-1$ (i.e., local components being 
Steinberg).  Here potential bias means there is a non-strict inequality of dimensions
when $M < N$.  However, one gets a strict bias (a strict inequality) when $M = N$, 
in which case the
parity condition agrees with the bias toward root number $+1$.  
Despite this potential bias, if $\frac NM$ is divisible by primes satisfying certain
congruence conditions, the sign patterns for $M$ are perfectly equidistributed in fixed 
spaces $S_k^\new(N)$.  One might think of these two phenomena as
saying sign patterns are equidistributed with a bias in the weight but 
perfectly equidistributed in the level.

Finally, in \cref{sec:gal-orb}, we give an explicit bound $K$ such that for any $k > K$
all sign patterns occur in $S_k^\new(N)$.  This gives a lower bound on the 
number of Galois
orbits in $S_k^\new(N)$. Conjecturally, this lower bound equals the number of Galois
orbits for sufficiently large $k$ (see \cite{tsaknias}).  Thus our bias of root numbers
suggests that Galois orbits tend to be slightly larger for newforms with root number 
$+1$.

\medskip
Our second motivation for obtaining these dimension formulas is to apply them 
to  the arithmetic of quaternion algebras and Eisenstein congruences to refine 
some results of \cite{me:cong}.  This will be treated in a separate paper.

We also suggest another possible use of one of our formulas: 
the dimension formula for $S_k^{\new,-}(N)$ tells us the dimension of the Saito--Kurokawa space of degree $2$ Siegel modular forms of paramodular level
$N$ and weight $\frac k2 + 1$, and thus may be useful when investigating paramodular
forms of squarefree level.

As one self-check for correctness, we compared our formulas for small weights and levels 
with known modular forms calculations via a combination of Sage and LMFDB.

\medskip
{\bf Acknowledgements.}
We are grateful to David Farmer, Abhishek Saha and Satoshi Wakatsuki for
helpful discussions and pointers to relevant literature.  
The author was partially supported by a Simons Foundation 
Collaboration Grant.

%
\section{Traces of Atkin--Lehner operators}
%

\label{sec:traces}

In this section, we give an explicit formula for the trace of a product of
Atkin--Lehner operators on the new space of even weight $k \ge 2$ and
squarefree level $N$.  
A formula on the full space of cusp forms $S_k(N)$ was originally given
by Yamauchi \cite{yamauchi} for arbitrary level (also with unramified Hecke
operators), though his final formula contains errors
(e.g., the first term on the right hand side of the statement of Theorem 1.6 
is missing the factor $(nN_0)^{1-k/2}$ coming from case (e) of $a(s)$ on p.\ 405).
Skoruppa and Zagier provide a corrected version in \cite{skoruppa-zagier}.
First we will state the corrected version in the case of squarefree level.

Throughout, $M$ denotes a divisor of our squarefree $N$,
$M' = N/M$, and we assume $M > 1$.  
For $p | N$, let $W_p$ denote the $p$-th Atkin--Lehner operator, and $W_M = \prod_{p | M} W_p$.
If $W$ is an operator on a vector space $S$, we denote its trace by $\tr_S W$
to clarify the vector space.

For $\Delta < 0$ a discrminant, let $h(\Delta)$ be the class number of an
order $\calO_\Delta$ of discriminant $\Delta$, and 
$w(\Delta) = \frac 12 |\calO_\Delta^\times|$.
Put $h'(\Delta) = h(\Delta)$ if $\Delta < -4$ but $h'(-4) = \frac 12$ and
$h'(-3) = \frac 13$.

Define 
\[ p_k(s) =
\begin{cases}
\frac{x^{k-1}-y^{k-1}}{x-y}& s \ne \pm 2 \\
k-1 & s = \pm 2 
\end{cases} \]
where $x, y$ are the roots of $X^2 - sX + 1$.
Put $r(D, n) = \# \{ r \mod 2n | r^2 \equiv D \mod 4n \}$
and let $\delta_{i,j}$ be the Kronecker delta function.

\begin{thm}[\cite{yamauchi}; \cite{skoruppa-zagier}] For squarefree $N$,
the trace $\tr_{S_k(N)} W_{M}$ equals
\begin{equation} \label{eq:SZ}
- \frac 12 \sum_{s} p_k(s/\sqrt{M})
\sum_{f} h'(\frac{s^2-4M}{f^2}) \sum_{t}
r(\frac{s^2-4M}{f^2 (M'/t)^2}, t) + \delta_{k,2}
\end{equation}
where $F = F_s \in \mathbb N$ is such that $(s^2-4M)/F^2$ is a fundamental discriminant.
Here $s$ runs over integers such that $s^2 < 4M$ and $M|s$,
$f$ runs over positive divisors of $F$ which are prime to $M$, and $t$ runs over positive divisors of $M'$ such that $\frac{M'}t | \frac{F}{f}$.
\end{thm}

This formula is already considerably simpler than in the case of non-squarefree
level, and next we will explicate it in a more elementary form.

We will use the following elementary facts about $r(D,n)$: $r(D,n)$ is multiplicative in 
$n$, and if $D$ is a discriminant then $r(D,p) = 1 + {D \leg p}$.

\medskip
First we consider the $s=0$ terms.  Note when $M > 3$, there is only an $s=0$
term in \eqref{eq:SZ}.

Assume $s=0$.  Then $p_k(s) = (-1)^{\frac k2 - 1}$,
and $F=2$ or $F=1$ according to whether $M$ is 3 mod 4 or not.
Then $t=M'$ except in the case that $F=2$, $f=1$ and $M'$ even which gives 
$t \in \{\frac {M'}2, M' \}$. 

If $M \equiv 1, 2 \mod 4$, then the $s=0$ summand in \eqref{eq:SZ} is simply
\begin{equation}
 (-1)^{\frac k2 - 1} h'(-4M) r(-4M, M').
\end{equation}
If $M \equiv 3 \mod 4$, then the $s=0$ summand is
\begin{equation}
 (-1)^{\frac k2 - 1} \left( h'(-4M) (r(-4M, M') + r(-M, \frac{M'}2)) + h'(-M) r(-M, M') \right),
\end{equation}
where we interpret $r(-M, \frac{M'}2) = 0$ if $M'$ is odd.  

Suppose $M \equiv 3 \mod 4$.  Then the well-known relation between the class number of a maximal order and a non-maximal order tells us $h'(-4M) = (2 - {M \leg 2} ) h'(-M)$.
Also note $r(-4M, M') = r(-M, M'_\odd)$, where $M'_\odd$ is the odd part of $M'$. 
Further, 
\[ r(-M, M') = r(-M, 2)  r(-M, \frac{M'}2)  = \left( 1 + {-M \leg 2} \right) r(-M, \frac{M'}2) \]
if $M'$ is even.

We can put all cases together as follows.  Let $a(M, M')$ be defined as follows:
\begin{center}
\begin{tabular}{c|c|c}
& $a(M, M')$ & $a(M, M')$ \\
$M \mod 8$ &  for $M'$ odd &  for $M'$ even \\
\hline
1, 2, 5, 6 & 1 & 1\\
3 & 4 & 6 \\
7 & 2 & 4
\end{tabular}
\end{center}
Let $\Delta_M$ be the discriminant of $\Q(\sqrt{-M})$.
Then the $s=0$ contribution to \eqref{eq:SZ} is
\begin{equation}
 (-1)^{\frac k2 - 1} a(M, M') h'(\Delta_M) r(-M, M'_\odd).
\end{equation}

\medskip
The only other terms in \eqref{eq:SZ} are the terms $s = \pm M$ when $M = 2, 3$.
We compute these now.  We first calculate
\[ p_k(\pm \sqrt 2) =
\begin{cases}
-1 & k \equiv 0 \mod 8 \\
1 & k \equiv 2 \mod 8 \\
1 & k \equiv 4 \mod 8 \\
-1 & k \equiv 6 \mod 8
\end{cases} \]
and
\[ p_k(\pm \sqrt 3) =
\begin{cases}
-1 & k \equiv 0 \mod 12 \\
1 & k \equiv 2 \mod 12 \\
2 & k \equiv 4 \mod 12 \\
1 & k \equiv 6 \mod 12 \\
-1 & k \equiv 8 \mod 12 \\
-2 & k \equiv 10 \mod 12.
\end{cases} \]
In these cases $s^2-4M$ is $-4$ or $-3$ as $M$ is $2$ or $3$, so $F=f=1$ and $t=M'$.
Thus each $s = \pm M$ summand of \eqref{eq:SZ} is 
\begin{equation}
p_k(\sqrt M) \frac 1M r(6-M, M').
\end{equation}

\medskip
Hence in all cases we may rewrite \eqref{eq:SZ} as
\begin{multline} \label{eq:SZ-rev}
\tr_{S_k (N)} W_M =\frac 12 (-1)^{\frac k2} a(M, M') h'(\Delta_M) r(\Delta_M, M'_\odd) \\
 - \frac 12 \delta_{M, 2} p_k(\sqrt 2)  r(-4, M') - \frac 13 \delta_{M, 3} p_k(\sqrt 3)  r(-3, M') + \delta_{k,2}.
\end{multline}

\medskip
To get the trace on the new space, we will use the following formula.
For $n \in \mathbb N$, let $\omega(n)$ be the number of prime divisors of $n$.

\begin{prop}[\cite{yamauchi}] For $N$ squarefree, $M | N$ and $M' = \frac NM$, we have
\[ 
\tr_{S_k^\new(N)} W_{M} = \sum_{d | M'} (-2)^{\omega(M'/d)} \tr_{S_k(dM)} W_{M}, 
\]
\end{prop}

We will compute this weighted sum over $d | M'$ of each term in \eqref{eq:SZ-rev}.
First note the binomial theorem implies 
\begin{equation} \label{eq:sum1}
\sum_{d | M'} (-2)^{\omega(M'/d)} = (-1)^{\omega(M')}.
\end{equation}

\begin{lem} Suppose $M'$ is odd and $D$ is a discriminant prime to $M'$.  Then
\[ \sum_{d | M'} (-2)^{\omega(M'/d)} r(D, d) = \prod_{p | M'} \left( {D \leg p} - 1\right) =
\begin{cases}
(-2)^{\omega(M')} & {D \leg p} = -1 \text{ for all } p | M' \\
0 & \text{else}.
\end{cases} \]
\end{lem}

\begin{proof} Let $M_1'$ be the product of $p | M'$ such that ${D \leg p} = -1$.
Then the above sum is
\[ \sum_{d | M'} (-2)^{\omega(M'/d)} \prod_{p | d} \left(1+{D \leg p} \right)
= (-2)^{\omega(M'/M_1')} \sum_{d | M_1'} (-2)^{\omega(M_1'/d)} 2^{\omega(d)}. \]
The sum on the right is $(-2 + 2)^{\omega(M_1')} = 0$ if $M_1' > 1$ and $1$ if $M_1' = 1$.  

(Alternatively, one can realize the sum as a Dirichlet convolution.)
\end{proof}

Noting that the only dependence of $a(M, M')$ on $M'$ is on the parity of $M'$, this
lemma combined with the above proposition is enough to get an explicit formula
for $\tr_{S_k^\new(N)} W_{M}$ when $M'$ is odd.

So let us consider the case that $M'$ is even.  To apply the previous lemma to this case,
we note that
\[ \sum_{d|M'} (-2)^{\omega(M'/d)} f(d) = (-2) \sum_{d|M'_\odd} (-2)^{\omega(M'_\odd/d)} f(d)
+ \sum_{d|M'_\odd} (-2)^{\omega(M'_\odd/d)} f(2d), \]
for a function $f$ on $\mathbb N$.  Combining this with \eqref{eq:sum1}, the lemma, and the
facts that $r(-4,2) = 1$ and $r(-3, 2) = 0$, when $M'$ is even
we see $\tr_{S_k^\new(N)} W_{M}$ is
\begin{multline*}
\frac 12 (-1)^{\frac k2} h'(\Delta_M) \left( -2 a(M,1) + a(M,2) \right)
\prod_{p | M'_\odd}  \left( {\Delta_M \leg p} -1 \right) \\
+ \frac 12 \delta_{M,2} p_k(\sqrt{2}) \prod_{p | M'_\odd} \left( {-4 \leg p} -1 \right) 
+ \frac 23 \delta_{M,3} p_k(\sqrt{3})  \prod_{p | M'_\odd} \left( {-3 \leg p} -1 \right) \\
+
(-1)^{\omega(M')} \delta_{k,2}.
\end{multline*}

We now summarize the formula for both cases, $M'$ odd and $M'$ even, together.  

Let $b(M, M') = a(M,1)$ when $M'$ is odd and $b(M, M') = -2a(M,1) + a(M,2)$
when $M'$ is even, i.e., $b(M, M')$ is given as follows:
\begin{center}
\begin{tabular}{c|c|c}
& $b(M, M')$ & $b(M, M')$ \\
$M \mod 8$ &  for $M'$ odd &  for $M'$ even \\
\hline
1, 2, 5, 6 & 1 & $-1$\\
3 & 4 & $-2$ \\
7 & 2 & 0
\end{tabular}
\end{center}

\begin{prop} \label{prop:newtrace}
Let $N$ be squarefree, $M | N$ with $M > 1$, $M' = N/M$ and $\Delta_M$
the discriminant of $\Q(\sqrt{-M})$.  Write $M' = 2^j M'_\odd$ with $M'_\odd$ odd and 
$j \in \{0, 1 \}$.  Then
\begin{multline*}
\tr_{S_k^\new(N)} W_{M} =
\frac 12 (-1)^{\frac k2} h'(\Delta_M) b(M, M') \prod_{p | M'_\odd}  \left( {\Delta_M \leg p} -1 \right)  \\
+ \delta_{k,2} (-1)^{\omega(M')} - \delta_{M,2} \frac{p_k(\sqrt{2})}2 \prod_{p | M'_\odd} \left( {-4 \leg p} -1 \right) \\
- \delta_{M,3} \frac {(-1)^j(j+1)}3  p_k(\sqrt{3})  \prod_{p | M'_\odd} \left( {-3 \leg p} -1 \right).
\end{multline*}
\end{prop}

The following bound will be useful for equidistribution results.

\begin{cor} \label{cor:tr-bound}
With notation as in the proposition, we have
\[ | \tr_{S_k^\new(N)} W_{M} | \le 2^{\omega(M'_\odd) + 1} h(\Delta_M) + \delta_{k, 2}. \]
\end{cor}

The next result will give us finer information for perfect equidistribution.

\begin{cor} \label{cor:tr0}
With notation as in the proposition, suppose $M > 3$.  

If $k \ge 4$, then 
$\tr_{S_k^\new(N)} W_{M} = 0$ if and only if one of the
following holds: (i) ${\Delta_M \leg p} = 1$ for some odd $p | M'$,
or (ii) $M'$ is even and $M \equiv 7 \mod 8$.

If $k=2$ and $\dim S_2(N) > 0$, then $\tr_{S_2^\new(N)} W_{M} = 0$ if and only if $N=M$ with $M \in \{ 37, 58 \}$ or $N=2M$ with $M \in \{ 13, 19, 37, 43, 67, 163 \}$.
\end{cor}

\begin{proof} The $k \ge 4$ case is evident.  For $k=2$, trace 0 can only happen
if exactly one of $h(\Delta_M)$, $|b(M, M')|$ and $\omega(M'_\odd)+1$ is $2$,
and the other two are $1$.   There is no $M > 2$ with $M \not \equiv 3 \mod 4$
such that $h(\Delta_M) = 1$, so we cannot have $\omega(M'_\odd) = 1$, and thus
$N \in \{ M, 2M \}$.  

If $b(M,M') = 2$ so $M \equiv 7 \mod 8$ and $N=M$, then $h(\Delta_M)= 1$ implies $M=7$ but $\dim S_2(7)=0$.  
If $b(M,M') = -2$, so $M \equiv 3 \mod 8$ and $M=2N$, then $h(\Delta_M) = 1$
implies $M \in \{ 11, 19, 43, 67, 163 \}$, but $\dim S_2(2M) = 0$ for $M = 11$.
The other 4 possibilities for $M$ all give trace 0 on nonzero spaces.

Lastly, if $|b(M,M')|=1$ so $M \not \equiv 3 \mod 4$ and $h(\Delta_M)=2$,
then $M$ is one of 5, 6, 10, 13, 22, 37 and 58.  One only gets nonzero newspaces
when $N=M$ and $M = 37, 58$ or $N=2M$ and $M=13, 37$. 
\end{proof}

%
\section{Refined dimension formulas I: plus and minus spaces}
%

\label{sec:pm}

As a first application, we consider the distribution of signs of function equations.
If $f \in S_k(N)$ is a newform, then
the sign of the functional equation $w_f$ of the $L$-series $L(s,f)$ 
is $(-1)^{k/2}$ times the eigenvalue of $W_N$.
Let $S_k^{\new, \pm}(N)$ be the subspace of $S_k(N)$ generated by newforms
$f$ with $w_f = \pm 1$.

For convenience, we recall the explicit formula for the full new space given by
G.\ Martin.
\begin{thm}[\cite{martin}] \label{thm:martin} For $N$ squarefree,
\begin{align*}
 \dim S_k^\new(N) = & \, \, \frac{(k-1)\varphi(N)}{12} + \left( \frac 14 + \lfloor \frac k 4 \rfloor
- \frac k4 \right) \prod_{p | N} \left( {-4 \leg p} - 1 \right) \\
& + \left( \frac 13 + \lfloor \frac k 3 \rfloor
- \frac k3 \right) \prod_{p | N} \left( {p \leg 3} - 1 \right) + \delta_{k,2} \mu(N).
\end{align*}
\end{thm}

Note 
\begin{equation} \label{eq:21}
\dim S_k^{\new, \pm }(N) = \frac 12 \left( \dim S_k^\new(N) \pm \tr_{S_k^\new(N)} \tr W_N \right).
\end{equation}
This combined with \cref{prop:newtrace} gives
the following explicit formulas for $\dim S_k^{\new, \pm}(N)$.  

\begin{thm} \label{thm:dimpm} Suppose $N > 1$ is squarefree.  When $N > 3$,
\[ \dim S_k^{\new, \pm}(N) = \frac 12 \dim S_k^\new(N) \pm
\frac 12 \left( \frac 12 h(\Delta_N) b(N, 1) - \delta_{k,2} (-1)^{\omega(N)} \right), \]
where we recall $\Delta_N$ is the discriminant of $\Q(\sqrt{-N})$ and
$b(N,1) = 1, 2$ or $4$ according to whether $N \not \equiv 3 \mod 4$,
$N \equiv 7 \mod 8$ or $N \equiv 3 \mod 8$.

When $N=2$ and $k > 2$,
\[ \dim S_k^{\new, \pm}(2) =
\frac 12 \dim S_k^\new(2) +
\begin{cases}
\pm \frac 12 & k \equiv 0, 2 \mod 8 \\
0 & \text{else}.
\end{cases} \]

When $N=3$ and $k > 2$,
\[ \dim S_k^{\new, \pm}(3) =
\frac 12 \dim S_k^\new(3) + 
\begin{cases}
\pm \frac 12 & k \equiv 0, 2, 6, 8 \mod 12 \\
0 & \text{else}.
\end{cases} \]
\end{thm}

We note that the $N > 3$ prime case of this result is essentially contained in 
\cite{wakatsuki} (also using \cite{yamauchi}, though the minus sign in Theorem 3.2 of 
\cite{wakatsuki}  should be ignored).

\begin{cor} \label{cor:pm-equi}
Fix $N > 1$ squarefree.   For any $k$,
we have $\dim S_k^{\new, +}(N) \ge \dim S_k^{\new, -}(N)$, and in fact
\[ \dim S_k^{\new, +}(N) - \dim S_k^{\new, -}(N) = c_N h(\Delta_N) - \delta_{k,2}(-1)^{\omega(N)} \]
where $c_N \in \{ \frac 12, 1, 2 \}$ is independent of $k$.  
In particular, as $k \to \infty$ along even integers we have 
\[ \dim S_k^{\new, \pm}(N) = \frac{(k-1)\phi(N)}{24} + O(1). \]
\end{cor}

The corollary says that the global
root numbers $w_f$ for fixed (squarefree) level are equidistributed as the weight
goes to infinity, but are biased toward $+1$ for bounded weights. 
Since the difference between $\dim S_k^{\new, +}(N)$ and 
$\dim S_k^{\new, -}(N)$ is a simple multiple of $h(\Delta_N)$ for $k \ge 4$, 
the bias roughly grows like $\sqrt N$ in the (squarefree) level (though as a proportion
of the total dimension, goes to zero as $k \to \infty$).
We can also get an asymptotic for varying the level replacing the $O(1)$ in the
last statement by $O(2^{\omega(N)})$ (cf.\ \cref{cor:sign-equi}).
This is a significant improvement on the asymptotic \cite[(2.73)]{ILS}.

\medskip
Next we consider for what spaces there is perfect or near perfect distribution of
the root numbers.

For $N=2, 3$, the behavior is clear from the theorem.  Namely, we always have
$\dim S_k^{\new, +}(N) - \dim S_k^{\new, -}(N) \in \{ 0, 1 \}$,
with $\dim S_k^{\new, +}(N) = \dim S_k^{\new, -}(N)$ if and only if $k \equiv 4, 6 
\mod 8$ when $N=2$, and if and only if $k \equiv 4, 10 \mod 12$ when $N=3$.

\begin{cor}
Suppose $N > 3$ is squarefree.  Then 
$\dim S_k^{\new, +}(N) = \dim S_k^{\new, -}(N)$ if and only if
$\dim S_k^\new(N) = 0$ or $k=2$ and $N \in \{ 37, 58 \}$.

When $k \ge 4$, we have that $\dim S_k^{\new, +}(N) = \dim S_k^{\new, -}(N) +1$
if and only if $N \in \{ 5, 6, 7, 10, 13, 22, 37, 58 \}$.

In general, for fixed $k$ and $r \in \Z_{\ge 0}$, there are only finitely many
squarefree $N$ such that $\dim S_k^{\new, +}(N) = \dim S_k^{\new, -}(N) + r$.
\end{cor}

\begin{proof} The first assertion follows from \eqref{eq:21} and
 \cref{cor:tr0}, and the second follows from the $k=2$ computations in the proof
 of said corollary.
\end{proof}

Hence perfect equidistribution of root numbers is very rare, which will be in contrast
to the situation for ``incomplete'' sign patterns in the next section.
Here are a couple of heuristic reasons for the bias towards root number $+1$
and thus the rareness of perfect equidistribution.

One reason comes from the philosophy that $L$-functions which ``barely exist''
tend to have negative signs, kindly explained to us by David Farmer.  
This is formulated in \cite{FK} where the focus is
on the second Fourier coefficient, but the same reasoning there applies to signs
of functional equations.  Namely, if $f \in S_k^\new(N)$, the explicit formula for
$L(s,f)$ expresses the sum of $\phi(\rho)$ in terms of $\log N$ and some other
terms, where $\phi$ is a suitable test function and $\rho$ runs over nontrivial
zeroes of $f$.  Taking $\phi$ to be a test function as in \cite[Thm 3.2]{FK} essentially
forces $N$ to be larger when $L(s,f)$ has a zero at the central point, in particular
when the root number is $-1$.  Thus one might expect root number $+1$ more often,
at least for small levels.  

Here is another reason coming from the arithmetic of quaternion algebras.  
Suppose $N=p$ is prime and $k=2$.  Let $B/\Q$ the definite quaternion algebra
of discriminant $p$.  Then the class number $h_B = 1 + \dim S_2^\new(p)$
and the type number $t_B = 1 + \dim S_2^{\new,+}(p)$.  Since $\frac 12 h_B \le t_B
\le h_B$, we see that $\dim S_2^{\new,+}(p)$ can vary between
$\frac 12 \dim S_2^\new(p) - \frac 12$ and $\dim S_2^\new(p)$.  This suggest
a bias toward root number $+1$, and one could compare type and class number
formulas to get another proof of \cref{thm:dimpm} when $N=p$ and $k=2$.
We will discuss this connection between dimension formulas
and arithmetic of quaternion algebras further in a subsequent paper.

%
\section{Refined dimension formulas II: local sign patterns}
%

\label{sec:sign-pat}

In this section, we obtain an
explicit formula for the space of newforms with fixed Atkin--Lehner eigenvalues
at primes dividing $M$, and obtain consequences for the distribution of this
collection of eigenvalues.

By a sign pattern $\eps_M$ for $M$,
we mean a multiplicative function $d \mapsto \eps_M(d)$ on the divisors $d$ of $M$
such that $\eps_M(1) = 1$ and $\eps_M(p) \in \{ \pm 1 \}$ for each $p | M$.
Define $S_k^{\new, \eps_M}(N)$ to be the subset of $S_k^\new(N)$ generated
by newforms $f$ such that the Atkin--Lehner eigenvalue of $f$ is $\eps_M(p)$
for each $p | M$.  

\begin{lem} Fix two sign patterns $\eps$, $\eps'$ for a squarefree $M$.  Then
\[ \sum_{d | M} \eps(d) \eps'(d) = \begin{cases}
2^{\omega(M)} & \eps = \eps' \\
0 & \text{else}.
\end{cases} \]
\end{lem}

\begin{proof} Let $S = \{ p | M : \eps(p) \ne \eps'(p) \}$.  The $\eps = \eps'$ case
is obvious, so assume $S \ne \emptyset$.  Note $\eps(d) \eps'(d) = 1$
if and only if the number of $p \in S$ such that $p | d$ is even.  Now 
precisely half the divisors $d$ of $M$ satisfy this property because exactly half the 
divisors of $\prod_{p \in S} p$ have an odd number of prime factors.
(If $\eps(M) = 1$ or $\eps'(M) = 1$, one can also realize this sum as a Dirichlet
convolution.)
\end{proof}

\begin{prop} \label{prop:sign-pat1}
Let $N$ be squarefree, $1 < M | N$, and $\eps_M$ a sign pattern for 
$M$.  Then
\[ \dim S_k^{\new,\eps_M}(N) = 2^{-\omega(M)} \sum_{d | M} \eps_M(d) \tr_{S_k^\new(N)} W_d, \]
where $W_1$ means the identity operator.
\end{prop}

\begin{proof}
Consider the sum on the right.  Note for any sign pattern $\eps'_M$ of $M$,
each term on the right gives a contribution of
$\pm \dim S_k^{\new, \eps'_M}(N)$.  The sign in the contribution is precisely
$\eps_M(d) \eps'_M(d)$.  Hence by the above lemma, the sum appearing
on the right is just 
$2^{\omega(M)}  \dim S_k^{\new,\eps_M}(N)$.
\end{proof}

This proposition combined with \cref{prop:newtrace} gives an explicit formula for
$\dim S_k^{\new,\eps_M}(N)$.  For simplicity, we just state it when there
are no extra $M=2$ or $M=3$ terms arising from \cref{prop:newtrace}.

\begin{thm} \label{thm:signpat}
Let $N$ be squarefree, $M | N$, 
and $\eps_M$ a sign pattern for $M$.  Assume for simplicity that 
(i) $2 \nmid M$ or ${-4 \leg p} = 1$ for some $p | \frac{N}{M}$, and
(ii) $3 \nmid M$ or ${-3 \leg p} = 1$ for some odd $p | \frac N{M}$.
Let $S$ be the set of divisors $d > 1$ of $M$ such that ${\Delta_d \leg p} = -1$
for all odd $p | \frac Nd$.  Put $\omega'(n) = \omega(n_\odd)$.
Then
\begin{multline*}
 \dim S_k^{\new,\eps_M}(N) \\
 =  \frac 1{2^{\omega(M)}} \Big( \dim S_k^\new(N) 
 + \frac 12(-1)^{\frac k2} \sum_{d \in S} \eps_M(d) h'(\Delta_d) b(d, N/d) (-2)^{\omega'(N/d)}  \\
   + \delta_{k,2} (-1)^{\omega(N)} \big( \prod_{p | M} (1-\eps_M(p)) - 1 \big)  \Big)
\end{multline*}
\end{thm}

The proposition and the theorem yield results about equidistribution of sign patterns.  

\begin{cor} \label{cor:sign-equi}
Let $M > 1$ be squarefree and $\eps_M$ a sign pattern for $M$.
As $kN \to \infty$ where $k$ is an even integer and $N$ is a squarefree multiple of $N$,
\[ \dim S_k^{\new,\eps_M}(N) = \frac {(k-1)\phi(N)}{2^{\omega(M)} \cdot 12}
+ O(2^{\omega(N)}). \]
\end{cor}

\begin{proof} Note that $2^{\omega(M)}$ times the right hand
side is an asymptotic for the full new space.  Now use \cref{prop:sign-pat1}
and note that naively summing
the bounds in \cref{cor:tr-bound} gives an error bound which is $O(2^{\omega(N)})$.
\end{proof}

Note this gives simultaneous equidistribution of sign patterns for fixed $M$ in both
weight and level, where the error term is $O(1)$ if we fix (or just bound) the level.
The error term can be made precise if desired.  We remark this implies the 
equidistribution part of \cref{cor:pm-equi} by taking $N=M$ and summing over sign patterns with $\eps_M(M)  +1$ or $-1$, but the explicit error term obtained in this way will be worse.

\medskip
Now if we fix the level $N$, we might ask if there is any bias in the collection of all possible
sign patterns, similar to the bias we saw for global root numbers.  Note if $N$ is prime,
then the sign patterns simply correspond to the global root numbers.  

So let us begin
by considering the case $N=pq$ for distinct primes $p, q > 3$ and take $k \ge 4$.
Then the relevant part of the formula in \cref{thm:signpat} for a sign pattern $\eps$ for
$N=M$ is 
\begin{equation} \label{eq:bias2}
 (-1)^{\frac k2} \left(-2 (\delta_{p \in S} \eps(p) h(\Delta_p) b(p, 1)  + \delta_{q \in S} \eps(q) h(\Delta_q) b(q, 1) )
+ \eps(pq) h(\Delta_{pq}) b(pq, 1) \right),
\end{equation}
where $\delta_{d \in S}$ is $1$ if $d \in S$ and $0$ otherwise.  Specifically, there
will be a bias towards (resp.\ away) from $\eps$, i.e., $\dim S_k^{\new, \eps}(N)$
is greater (resp.\ less) than $2^{-\omega(N)} \dim S_k^\new(N)$ if and only if 
\eqref{eq:bias2} is positive (resp.\ negative).

For simplicity, assume $k \equiv 0 \mod 4$, so the biases we describe will be flipped
if $k \equiv 2 \mod 4$.  We write the possible sign patterns $\eps$ for $N=pq$ as $++$, $+-$,
and so on, where the first sign is the sign of $\eps(p)$ and the second is the sign of 
$\eps(q)$.  

If $p, q \not \in S$, then only the last term of \eqref{eq:bias2}
appears, and hence we get an equal bias toward each of $++$ and $--$ and the same bias 
away from each of $+-$ and $-+$.  

If $p \in S$ but $q \not \in S$, then \eqref{eq:bias2}
will be positive and maximized when $\eps(p)=\eps(q)=-1$, so there is a definite bias towards
$--$, and similarly a bias away from $+-$.  For the signs $++$ and $-+$, the actual
values of class numbers and $b(-, 1)$ come into play, and it is not clear which has a positive
or negative bias, but at least we can say the bias will not be as strong for $--$ and $+-$.

The case of $p \not \in S$, $q \in S$ is similar so we lastly suppose both $p, q \in S$.
As in the previous case, there is a clear and maximal bias toward the sign pattern $--$,
though the bias to or away from any of the other sign patterns depends on class numbers
and congruences of divisors.

So while there does not appear to be any nice uniform description of the biases for
general sign patterns, there is at least a clear bias for $--$ in all cases.  The same argument
generalizes to the following.

\begin{cor} \label{cor:bias2}
 Let $k \ge 4$, $M, N$ squarefree odd with $M | N$ and $M > 3$.  
If $M$ is divisible by $3$, assume there is an odd $p | \frac NM$ such that ${-3 \leg p} = 1$. 
Let $-_M$ be the sign pattern for $M$ given by $-_M(p) = -1$ for each $p | M$.
Then for any other sign pattern $\eps$ for $M$, we have
\begin{align*}
 \dim S_k^{\new, -_M}(N)  &\ge \dim S_k^{\new, \eps}(N) \quad \text{if } \frac k2 + \omega(N) \text{ is even}, \\
  \dim S_k^{\new, -_M}(N)  &\le \dim S_k^{\new, \eps}(N) \quad \text{if } \frac k2 + \omega(N) \text{ is odd}. 
\end{align*}
Moreover, if $M=N$ the above inequalities are strict for at least one choice of $\eps$.
\end{cor}

\begin{proof} The above assumptions guarantee that each term in the sum over $d \in S$
in \cref{thm:signpat} has sign $(-1)^{\omega(N)}$.  If $M=N$, we always have $N \in S$
so this sum is nonzero.
\end{proof}

We note this bias to or away from $-_N$ agrees with the bias toward global 
root number $+1$.  For example, suppose $M=N=35$.  In weight 4 there are 3 Galois orbits,
one of size 1 with signs $+-$, one of size 2 with signs $++$, and one of size 3 with signs $--$
(root number $+1$).
In weight 6, there are 4 Galois orbits, one for each possible sign pattern,  with sizes 1, 2, 3, 4 and  the smallest orbit has signs $--$ (root number $-1$).

When $M=N$ is a product of an odd number of factors with $M, N$ as in the corollary, 
the same argument also gives a bias toward $-_N$ when $k=2$.  

\medskip
On the other hand, the theorem also gives us perfect equidistribution of sign patterns 
when moving to to ``sufficiently large'' level.

\begin{cor} Fix an even $k \ge 4$.
Let $M > 1$ be squarefree and $\eps_M, \eps_M'$ be any two sign patterns
for $M$.  Let $S$ be a set of odd primes $p \nmid M$ such that for any $d | M$,
${\Delta_d \leg p} = 1$ for some $p \in S$.  (If $M$ is odd, we can omit the $d=1$
case of this condition.) Then for any squarefree $N$ divisible by both
$M$ and each $p \in S$, we have
\[ \dim S_k^{\new, \eps_M}(N) = \dim S_k^{\new, \eps_M'}(N). \]
\end{cor}

For instance, let $M=10$ and $S = \{ 3, 13 \}$.  Note ${-4 \leg 13} = {-40 \leg 13} = 1$
and ${-8 \leg 3} = {-20 \leg 3} = 1$.  Hence for any squarefree $N$ which is a multiple of
$390$, all sign patterns occur equally often for the Atkin--Lehner eigenvalues at $2$ and
$5$ among the newforms of level $N$ and a fixed weight $k \ge 4$.

Put another way, the corollary says that if the primes dividing $\frac NM$ satisfy
certain congruence conditions, then we have perfect equidistribution of
sign patterns for $M$.  So if we think about starting with a fixed $M$ and $k$, and 
successively raising the level by randomly adding other prime factors, then with
probability $1$ we will eventually reach a state of perfect equidistribution.  Thus
we may think of this as saying there is perfect equidistribution in the level.  
Note that even
though \cref{cor:bias2} gives a bias to or away from $-_M$, each time we add
a prime to the level in this process, we flip the sign of this bias, 
so even before we hit a fixed level $N$ with perfect equidistribution, variation in the distribution of the sign patterns appears to oscillate.

%
\section{Bounds on number of Galois orbits}
%

\label{sec:gal-orb}

Let $f \in S_k^\new(N)$ be a newform.  By the sign pattern for $f$, we mean a sign pattern 
$\eps$ for $N$ such that $\eps(p)$ is the eigenvalue of the $p$-th Atkin--Lehner operator for 
$f$, for all $p | N$.  We say $\eps$ occurs in $S_k^\new(N)$ if it is the sign pattern of some 
newform  $f \in S_k^\new(N)$.  More generally, if $\eps_M$ is a sign pattern for $M | N$, we 
say it occurs in $S_k^\new(N)$ if $S_k^{\new, \eps_M}(N) \ne 0$.

Clearly, if two new forms $f, g \in S_k^\new(N)$ have different sign patterns, then $f, g$
lie in different Galois orbits.  Thus the number of sign patterns occurring in $S_k^\new(N)$
provides a lower bound on the number of Galois orbits.  A generalization of
Maeda's conjecture asserts that for squarefree level and sufficiently large weight, 
the number of Galois orbits is precisely the number of possible sign patterns, 
$2^{\omega(N)}$ (\cite{tsaknias}; see also \cite{chow-ghitza}).  
One consequence of our results gives an effective bound on weights with at least this many 
Galois orbits.   

However, we remark that in some low weights
the number of Galois orbits is strictly larger than the number of sign patterns which
occur---e.g., there are 2 Galois orbits in $S_6^\new(17)$ with Atkin--Lehner eigenvalue
$+1$ at $17$.  Hence, even admitting the truth of this 
generalized Maeda conjecture, one may
need to take still higher weights to get exactly $2^{\omega(N)}$ orbits.

We also note that the generalized Maeda conjecture together with our results on bias
of signs would imply that there is a bias in the distribution of the size of Galois orbits 
when separating by root number or sign patterns.  In particular, the average
size of Galois orbits of newforms with root number $+1$ should be larger
than that for newforms with root number $-1$ for fixed $N, k$, though these averages
should be asymptotic as $kN \to \infty$.  In fact, looking at tables in LMFDB 
for small $N$ suggests there may be a tendency for Galois orbits with
root number $+1$ to be larger on average already in weight $k=2$.

\begin{prop} Fix $M | N$ squarefree with $M > 1$.  
Let $H_M = \max \{ h(\Delta_d) : d | M \}$ and
\[ K_{N,M}=  \frac{24(3^{\omega(N)} - 2^{\omega(N_\odd)}) H_M + 10 \cdot 2^{\omega(N)}}{\phi(N)} + 1. \]
Then for any even $k > \min \{ K_{N, M}, 3 \}$, all possible sign patterns for $M$ occur
for $S_k^\new(N)$, and thus there are at least $2^{\omega(M)}$ Galois orbits in
$S_k^\new(N)$.
\end{prop}

\begin{proof} We just need to show that the $d=1$ term in \cref{prop:sign-pat1} for
$M=N$ is larger in absolute value than the sum of all the other terms with $d | N$.
Now compare \cref{thm:martin} and \cref{cor:tr-bound}.  Here we majorized the
sum $\sum_{d | M} 2^{\omega'(N/d)}$ with the same sum taken over $d | N$.
\end{proof}

We did not strive for optimality with this bound.  For instance,
one can improve the bound for $N$ along various families 
by working with a set $S$ as in the statement of \cref{thm:signpat}.  

When $N=M$, this gives a lower bound on the weight to get all possible sign patterns
and as $N \to \infty$, note that $K_{N,N}$ grows at a slower rate than $H_N$, which grows roughly at a rate of $\sqrt{N}$.  On the other hand, for fixed $M$ and $N \to \infty$,
$K_{N, M} \to 1$, giving all sign patterns for a fixed weight and large level.  
(A simple modification of the bound will treat $k=2$.)
This agrees with the equidistribution results on sign patterns in \cref{cor:sign-equi} for a fixed
weight and varying level $k \to \infty$.

In the simple case of prime level, we get the following improved explicit bounds.

\begin{prop} For $p \ge 13$ both sign patterns for $p$ occur in $S_k^\new(p)$ for
all $k \ge 4$.  When $p \in \{ 7, 11 \}$ (resp.\ $p=5$), the same is true for $k \ge 6$
(resp.\ $k \ge 8$).  

When $k=2$, both sign patterns for $p$ occur in $S_k^\new(p)$ if and only if
$p > 60$ and $p \ne 71$, or $p=37$.
\end{prop}

\begin{proof} From \cref{thm:dimpm}, both sign patterns occur in $S_k^\new(p)$
whenever
\[ k > \frac{6 b(p,1){h(\Delta_p)} + 20}{p-1} + 1, \]
for $k\ge 4$ and $p > 3$.
Now the class number formula combined with an explicit bound on Dirichlet $L$-values
(see \cite[Prop 10.3.16]{cohen:ntii}) tells us $h(\Delta_p) \le \frac{\sqrt p}{\pi} (\frac 12
\log p + \log \log p + 3.5)$.  Since $b(p, 1) \le 4$,
the above bound on $k$ using this estimate
is less than 4 for $p > 157$ and less than 2 for $p > 2575$.  
Using exact class number calculations (and $b(p, 1)$ 
rather than 4), we see the above bound on $k$ is less than $4$
for all $p > 11$, and less than 2 for all $p > 60$ except $p \in \{ 71, 79, 83, 89, 101,
131 \}$.  Explicit calculations finish the $k=2$ case.
For $p = 5, 7, 11$ this bound respectively gives the result
for $k \ge 10$, $k \ge 8$, $k \ge 6$.  Explicit calculation of these spaces
then shows the stated bounds on $k$ hold (and are optimal) for $p \in \{ 5, 7, 11 \}$.
\end{proof}

Note for $p = 2, 3$, \cref{thm:dimpm} implies both signs occur in $S_k^\new(p)$ 
whenever 
$\dim S_k^\new(p) \ge 2$.  When $p=2$, this happens for $k \in \{ 14, 20, 22 \}$
or $k \ge 26$.  When $p=3$, this happens for $k=10$ or $k \ge 14$.

%
%

\end{document}